\documentclass[12pt]{article}
 \usepackage{latexsym}
 \usepackage{amsmath,amsfonts}
 \usepackage{amsthm,amssymb}
 \usepackage{mathrsfs}
\usepackage{bbm}
\usepackage{stmaryrd}
\usepackage[T1]{fontenc}
 \title{On two multistable extensions of stable L\'evy\\ motion and their semi-martingale representations}
 \author{Ronan Le Gu\'{e}vel$^*$, Jacques L\'evy V\'ehel$^{**}$ and Lining Liu$^{**}$}

 \date{
 }

 \newtheorem{theo}{Theorem}
 \newtheorem{prop}[theo]{Proposition}
 
 \newtheorem{lem}[theo]{Lemma}
 \newtheorem{cor}[theo]{Corollary}

\newcommand\bbbr{\mathbb{R}} 
\newcommand\bbbn{\mathbb{N}} 

\renewcommand\P{\mathbb{P}}

 \newcommand\ed{\stackrel{d}{=}}

\newcommand\X{{\sf X}}
\newcommand\Y{{\sf Y}}
\newcommand{\1}{{\bf 1}}
\newcommand\E{\mathbb{E}}

\def\keywordname{{\bf Keywords:}}
\newcommand{\keywords}[1]{\par\addvspace\baselineskip\noindent\keywordname\enspace\ignorespaces#1}

\begin{document}
 
 \maketitle
 
\begin{abstract}
\par \noindent We study two versions of multistable L\'evy motion. 
Such processes are extensions of classical
L\'evy motion where the stability index is allowed to vary in time, a useful 
property for modelling non-increment stationary phenomena. We show that 
the two multistable L\'evy motions have distinct properties: in particular, one is a  
pure-jump Markov process, while the other one satisfies neither of these properties. We prove
that both are semi-martingales and provide semi-martingale decompositions. 
\end{abstract}



\bigskip
\begin{footnotesize}
\noindent
$^*$Universit\'{e} de Rennes 2 - Haute Bretagne, Equipe de Statistique Irmar,
UMR CNRS 6625, Place du Recteur Henri Le Moal, CS 24307, 35043 RENNES Cedex, France

\noindent
$^{**}$Regularity Team, Inria and MAS Laboratory, Ecole Centrale Paris - Grande Voie des Vignes, 92295 Ch\^atenay-Malabry Cedex, France

\noindent
ronan.leguevel@univ-rennes2.fr, jacques.levy-vehel@inria.fr, liningliu.422@gmail.com
\end{footnotesize}

\section{Background on multistable L\'evy motions}

The class of {\it multistable processes} was introduced in \cite{FL} and
further developed and studied in \cite{ AA,BL,FLL,KL,LL,LL2}. These processes extend the well-known stable processes (see, {\it e.g.} \cite{Bk_Sam}) by letting
the stability index $\alpha$ evolve in ``time''. Such models are useful in various applications where the data display jumps with varying intensity, 
such as financial records, EEG or natural terrains:
indeed, multistability is one practical way to deal with  
non-stationarities observed in various real-world
phenomena. Generally speaking, non-stationarity is a "non-property", 
and as such, is hard to model directly. One approach to tackle 
this difficulty is to consider processes which are {\it localisable} \cite{F,F2}.
Recall that a process $Y =\{Y(t): t \in \bbbr\}$
is $h-$localisable at $u$ if there exists an $h \in \bbbr$
and a non-trivial limiting process $Y_{u}'$ such that
\begin{equation}
\lim_{r \to 0}\frac{Y(u+rt) -Y(u)}{r^{h}} = Y_{u}'(t).
\label{locform1}
\end{equation}
(Note $Y_{u}'$ will in general vary with $u$.) When the limit exists,
$Y_{u}'=\{Y_{u}'(t): t\in \bbbr\}$ is termed the {\it local form} or tangent
process of $Y$ at $u$.
The limit (\ref{locform1}) may
be taken either in finite dimensional distributions or distribution (one then speaks of strong localisability).
A classical example of a localisable process is multifractional Brownian motion
$Z$ which ``looks like'' index-$h(u)$ fractional Brownian motion (fBm) close to time $u$ but where $h(u)$ varies, that is
\begin{equation*}
\lim_{r \to 0}\frac{Z(u+rt) -Z(u)}{r^{h(u)}} = B_{h(u)}(t) \label{exfbm}
\end{equation*}
where $B_{h}$ is index-$h$ fBm.
A generalization of mBm, where the Gaussian
measure is replaced by an $\alpha$-stable one, leads to
multifractional stable processes \cite{ST1}.

The $h$-local form $Y_{u}'$ at $u$, if it exists,
must be $h$-self-similar, that is $Y_{u}'(rt)\ed r^{h}Y_{u}'(t)$
for $r>0$. In addition, as shown in \cite{F,F2}, under quite general
conditions, $Y_{u}'$ must also have stationary increments at almost all $u$ at
which there is strong localisability.  Thus, typical local forms
are self-similar with stationary increments (sssi), that is
$r^{-h}(Y_{u}'(u+rt) -Y_{u}'(u)) \ed Y_{u}'(t)$ for
all $u$ and $r>0$. Conversely, all sssi processes are localisable.
Classes of known sssi processes include fBm,
linear fractional stable motion and $\alpha$-stable L\'{e}vy
motion, see \cite{Bk_Sam}.

Localisability provides a convenient way to handle non stationarity: an everywhere localisable process
may heuristically be seen as a "gluing" of tangent "simple" increment-stationary processes, where simple, in this frame,
means sssi. 
Multistable processes are instances of this idea with the tangent processes being L\'{e}vy
stable ones, and the characteristic varying in time being the stability exponent $\alpha$. Although the tangent 
processes are semi-martingales in this case, there is no reason why the original process will share this property. The main
contribution of this work is to show that, for a simple class of multistables processes, this will indeed be the case.
As a consequence, such multistable processes are well fitted for applications e.g. in financial modelling or traffic engineering.

Three paths have been explored so far to define multistable processes: the first one uses a {\it field of stable processes} $X(t, \alpha)$, and obtains a multistable process by considering a ``diagonal'' $Y(t)=X(t,\alpha(t))$ on this field. 
This is the approach 
of \cite{FL,LL}. In \cite{FLL}, multistable processes are obtained from {\it moving average processes}. 
Finally, in \cite{KL}, {\it multistable measures} are used for this purpose. 

In this work, we shall be interested only in multistable L\'evy motions, which are the simplest examples of 
multistable processes.  These multistable extensions of $\alpha-$stable L\'evy motion were constructed in \cite{FL,LL} 
using respectively Poisson and Fergusson-Klass-LePage representations, while related processes were defined 
in \cite{KL} by their characteristic functions.
In order to give precise definitions, let us set the following notation, which will be used throughout the paper:
\begin{itemize}
	\item $\alpha: \bbbr \to [c,d] \subset (0,2)$ is a continuous function.
	\item $\Pi$ is a Poisson point process on $[0,T] \times \bbbr$, $T>0$, with mean measure the Lebesgue measure $\mathcal{L}$.
	\item $(\Gamma_i)_{i \geq 1}$ is a sequence of arrival times of a Poisson process with unit arrival time.
	\item $(V_i)_{i \geq 1}$ is a sequence of i.i.d. random variables with uniform distribution on $[0,T]$.
	\item $(\gamma_i)_{i \geq 1}$ is a sequence of i.i.d. random variables with distribution
	$P(\gamma_i=1)=P(\gamma_i=-1)=1/2$.
\end{itemize}
The three sequences $(\Gamma_i)_{i \geq 1}$, $(V_i)_{i \geq 1}$, and $(\gamma_i)_{i \geq 1}$ are independent.

\bigskip

The following process, that we shall call 
{\it field-based multistable L\'evy motion}, is considered in \cite{FL}:
\begin{equation} \label{FBP}
L_{F}(t)  = C_{ \alpha(t)}^{1/ \alpha(t)}\sum_{(\X,\Y)\in\Pi}
\1_{[0,t]}(\X)\Y^{<-1/\alpha(t)>} \quad (t \in [0,T]),
\end{equation}
where  $Y^{<-1/\alpha(t)>} := \mbox{sign}(Y)|Y|^{-1/\alpha(t)}$ and
\begin{equation}\label{calphax}
C_{u}= \left( \int_{0}^{\infty} x^{-u} \sin (x)dx \right)^{-1}.
\end{equation}
Note that, when $\alpha(t)$ equals the constant $\alpha$ for all $t$, $L_{F}$ is simply the Poisson representation of 
$\alpha-$stable L\'evy motion, that we denote by $L_{\alpha}$. If one considers
instead the Fergusson-Klass-LePage representation of L\'evy motion, then it is natural to define $L_{F}$ as follows:

\begin{equation} \label{FBFKL}
L_{F}(t)  =  C_{ \alpha(t)}^{1/ \alpha(t)} T^{1/ \alpha(t)}\sum_{i=1}^{+ \infty}
\gamma_i \Gamma_i^{-1/\alpha(t)} 1_{[0,t]}(V_i) \quad (t \in [0,T]).
\end{equation}

This is the approach of \cite{LL}, where it is proven that \eqref{FBP} and \eqref{FBFKL} indeed define the same process, and that 
the joint characteristic function of $L_{F}$ equals:

\begin{equation}\label{FBCF}
\E \left( \exp \left(i \sum_{j=1}^{m}\limits \theta_j L_{F}(t_j) \right) \right)= \exp \left( -2 \int_{[0,T]} \int_{0}^{+ \infty} \sin^2\left( \sum_{j=1}^{m} \theta_j \frac{C_{\alpha(t_j)}^{1/\alpha(t_j)}}{2y^{1/\alpha(t_j)}} 1_{[0,t_j]}(x) \right)\hspace{0.1cm} dy \hspace{0.1cm} dx \right)
\end{equation}
where $ m \in \mathbb{N}, (\theta_1, \ldots, \theta_m) \in \bbbr^m, (t_1, \ldots , t_m) \in \bbbr^m$.

\bigskip
In \cite{KL}, another path is followed to define a multistable extension of L\'evy motion.
Considering the characteristic function of $\alpha-$stable L\'evy motion $L$:
\begin{equation*}
\mathbb{E}\left(\exp(i\theta L(t))\right)= \exp \left(-t |\theta|^{\alpha} \right),
\end{equation*}
one defines the process $L_{I}$ by its joint characteristic function as follows:
\begin{equation}\label{IICF}
\mathbb{E}\left(\exp\left(i\sum_{j=1}^d\theta_j L_{I}(t_j)\right)\right)=\exp\left(-\int\left|\sum_{j=1}^d \theta_j1_{[0,t_j]}(s)\right|^{\alpha(s)} ds\right).
\end{equation}

Both $L_{F}$ and $L_{I}$ are tangent, at each time $u$, to a stable L\'evy motion with exponent $\alpha(u)$: 
\begin{equation*}
\frac{Y(u+rt) - Y(u)}{r^{\frac{1}{\alpha(u)}}} \to L_{\alpha(u)}(t) 
\end{equation*}
as $r \searrow 0$, where $Y=L_{F}$ or $Y=L_{I}$, and convergence is of finite dimensional
distributions. As discussed above, they thus provide flexible models to study phenomena 
which locally look like stable L\'evy motions, but where the stability index would evolve in time.

It is clear from this definition that $L_{I}$ has independent increments (this is why we call this version 
the {\it independent increments multistable L\'evy motion}). This is a strong difference with $L_{F}$, which is not even a 
Markov process. The construction based on multistable measures thus offers the advantage of retaining some 
properties of classical L\'evy motion: 
obviously the increments cannot be stationary any more, but we still deal 
with a pure jump additive process. This is not the case for the field-based L\'evy motion (which is not a pure jump
process, see Section 3). However, a drawback of \eqref{IICF} is that, while $L_{F}$ coincides, 
at each fixed time, with a L\'evy motion, this is not true of $L_{I}$.

The main aim of this work is to prove that both $L_{I}$ and $L_{F}$ are semi-martingales. This is almost obvious
of $L_{I}$, but far less so of $L_{F}$ (we prove in fact a more general statement, see Theorem \ref{semimgal}).
This property is of crucial importance for the very applications for which
multistable processes were designed in the first time, in particular financial modelling: in \cite{LCLV},
we use this property to analyse a "stochastic local jump intensity" model, 
much in the spirit of classical stochastic volatility models, but where the parameter $\alpha$ 
rather than the volatility evolves in a random way.

The remainder of this work is organized as follows: Section \ref{serrep} elucidates the links between $L_{F}$ and $L_{I}$ 
and provides series representations of $L_{I}$. We prove that both processes are semi-martingales and give their semi-martingales decompositions in Section \ref{semidec}.

\section{Series representations of independent increments multistable L\'{e}vy Motion}\label{serrep}

While three characterizations of $L_{F}$ are known, namely the Poisson and Fergusson-Klass-LePage representations and its characteristic function, 
only the latter is available for $L_{I}$. In this section, we provide the series representation of independent increments multistable L\'evy motion.

\subsection*{Poisson series representation}
\begin{prop}\label{propPRLII}
$L_{I}$ admits the following representation in law for $ t\in [0,T]$:
\begin{equation}\label{PRLII}
L_{I}(t)=\sum_{(X,Y)\in \Pi}C_{\alpha(X)}^{1/ \alpha(X)}\mathbf{1}_{[0,t]}(X)Y^{<-1/\alpha(X)>}.
\end{equation}
\end{prop}

\begin{proof}
The proof is rather routine. Call $\tilde L$ the process on the right hand side of 
\eqref{PRLII}. Proposition 4.2 in \cite{FL} implies that the series converges, and that its marginal 
characteristic function reads:
\begin{equation*}\label{charcyt}
\mathbb{E}\left(\exp(i\theta \tilde L(t))\right)= \exp \left(-2\int \int \sin^{2}\left(\frac{1}{2}\theta C_{\alpha(s)}^{1/ \alpha(s)}\mathbf{1}_{[0,t]}(s)|y|^{-1/\alpha(s)}\right)dsdy\right),
\end{equation*}
(an integral sign without bounds means integration over the whole domain).
We first prove that this quantity is equal to $\exp \left(-\int_0^{t} |\theta|^{\alpha(s)} ds\right)$. One computes:
\begin{eqnarray*}
\nonumber\int 2 \sin^{2}\left(\frac{1}{2}\theta C_{\alpha(s)}^{1/ \alpha(s)}\mathbf{1}_{[0,t]}(s)|y|^{-1/\alpha(s)}\right)dy&=& \int \left(1-\cos (\theta C_{\alpha(s)}^{1/ \alpha(s)} \mathbf{1}_{[0,t]}(s)|y|^{-1/\alpha(s)})\right)dy\\
\nonumber &=& \alpha(s)|\theta|^{\alpha(s)} C_{\alpha(s)} \mathbf{1}_{[0,t]}(s) \int \frac{(1-\cos( z) )}{|z|^{\alpha(s)+1}}dz\\
\label{integrand} &=& |\theta\mathbf{1}_{[0,t]}(s)|^{\alpha(s)},
\end{eqnarray*}
where we have used the change of variables $z=\theta C_{\alpha(s)}^{1/ \alpha(s)} \mathbf{1}_{[0,t]}(s)y^{-1/\alpha(s)}$ 
and the fact that 
$$\alpha(s) \int \frac{(1-\cos( z) )}{|z|^{\alpha(s)+1}}dz = \frac{1}{C_{\alpha(s)}}.$$
Thus,
\begin{eqnarray*}
\mathbb{E}\left(\exp(i\theta \tilde L(t))\right)&=& \exp \left(-\int \int 2\sin^{2}\left(\frac{1}{2}\theta C_{\alpha(s)}^{1/ \alpha(s)}\mathbf{1}_{[0,t]}(s)y^{-1/\alpha(s)}\right)dsdy\right)\\
&=& \exp \left(-\int |\theta\mathbf{1}_{[0,t]}(s)|^{\alpha(s)} ds\right)\\
&=&\exp \left(-\int_0^t|\theta|^{\alpha(s)} ds\right).
\end{eqnarray*}

The process defined by the series on the right hand side of \eqref{PRLII} clearly has independent increments. This fact and the above computation
ensures that it coincides in law with $L_{I}$.
\end{proof}

\subsection*{Fergusson-Klass-Le Page series representation}
To simplify the notation, we set $T=1$ in the next proposition.

\begin{prop}\label{propFKLLII}
$L_{I}$ admits the following representation in law:
\begin{equation}\label{FKLLII}
L_{I}(t)=\sum_{i=1}^{\infty} C_{\alpha(V_i)}^{1/ \alpha(V_i)} \gamma_i \Gamma_i^{-1/\alpha(V_i)}\mathbf{1}_{(V_i \leq t)}.
\end{equation}
In particular, the marginal characteristic function of the right hand side $Y$ 
of \eqref{FKLLII} is
\begin{equation}\label{mcfii}
\mathbb{E}\left(\exp(i\theta Y(t))\right)= \exp \left(-\int_0^{t} |\theta|^{\alpha(u)} du\right).
\end{equation}
\end{prop}

\begin{proof}
Again, the proof is simple.
Let us first verify that the marginal characteristic function of the right hand 
side $Y$ of \eqref{FKLLII} is indeed given by \eqref{mcfii}.
Set 
$$f(X, m_X)=C_{\alpha(U)}^{1/ \alpha(U)} \gamma X^{-1/\alpha(U)}\mathbf{1}_{(U\leq t)}$$
where $m_X=(\gamma, U)$. Using the Marking Theorem (see \cite[p. 55]{JK}), since $\mathbb{P}(\gamma=1)=\mathbb{P}(\gamma=-1)=1/2$, one computes: 
\begin{eqnarray}
\nonumber \mathbb{E}\left(\exp(-i\theta Y(t))\right)&=& \exp \left( \int \int (1- e^{-f(x,m)})dx \mathbb{P}(x,dm)\right)\\
\nonumber &=&\exp \left( \int \int (1- e^{-i\theta C_{\alpha(u)}^{1/ \alpha(u)} \gamma x^{-1/\alpha(u)}\mathbf{1}_{(u\leq t)}})dx \mathbb{P}(x,dm)\right)\\
\nonumber&=& \exp\left( -\frac{1}{2} \int_0^{\infty}\int_0^{\infty}\left(\left(1-e^{i\theta x^{-1/\alpha(u)}\mathbf{1}_{(u\leq t)}}\right)+\left(1-e^{-i\theta x^{-1/\alpha(u)}\mathbf{1}_{(u\leq t)}}\right)\right)dxdu\right)\\
\nonumber&=& \exp\left(-\int_0^{\infty}\int_0^{\infty}\left(1-\frac{e^{i\theta x^{-1/\alpha(u)}\mathbf{1}_{(u\leq t)}}+e^{-i\theta x^{-1/\alpha(u)}\mathbf{1}_{(u\leq t)}}}{2}\right)dxdu\right)\\
\label{previouslemma}&=& \exp\left(-\int_0^{\infty}\int_0^{\infty}(1-\cos (\theta C_{\alpha(u)}^{1/ \alpha(u)} \mathbf{1}_{(u\leq t)}x^{-1/\alpha(u)}))dxdu\right)\\
\label{serieschac}&=& \exp \left(-\int_0^{t} |\theta|^{\alpha(u)} du\right)
\end{eqnarray}
where the passage from (\ref{previouslemma}) to (\ref{serieschac}) follows along the same lines as in the proof of Proposition \ref{propPRLII}.
Now,
\begin{eqnarray*}
\hspace{-1.8cm}
\mathbb{E}\left(\exp(-i\theta \left(Y(t)-Y(s)\right))\right)&=&\exp \left( \int \int (1- e^{-i\theta C_{\alpha(u)}^{1/ \alpha(u)} \gamma x^{-1/\alpha(u)}\mathbf{1}_{(s<u\leq t)}})dx \mathbb{P}(x,dm)\right)\\
&=& \exp\left( -\frac{1}{2} \int_0^{\infty}\int_0^{\infty}\left(\left(1-e^{i\theta x^{-1/\alpha(u)}\mathbf{1}_{(s<u\leq t)}}\right)+\left(1-e^{-i\theta x^{-1/\alpha(u)}\mathbf{1}_{(s<u\leq t)}}\right)\right)dxdu\right)\\
&=& \exp\left(-\int_0^{\infty}\int_0^{\infty}\left(1-\frac{e^{i\theta x^{-1/\alpha(u)}\mathbf{1}_{(s<u\leq t)}}+e^{-i\theta x^{-1/\alpha(u)}\mathbf{1}_{(s<u\leq t)}}}{2}\right)dxdu\right)\\
&=& \exp\left(-\int_0^{\infty}\int_0^{\infty}(1-\cos (\theta C_{\alpha(u)}^{1/ \alpha(u)} \mathbf{1}_{(s<u\leq t)}x^{-1/\alpha(u)}))dxdu\right)\\
&=& \exp \left(-\int_0^{\infty} | \mathbf{1}_{(s,t]}(u)\theta|^{\alpha(u)} du\right)\\
&=& \exp \left(-\int_s^{t} |\theta|^{\alpha(u)} du\right).
\end{eqnarray*}
This implies independence of the increments and ends the proof.
\end{proof}

\section{Multistable L\'evy motions are semi-martingales}\label{semidec}
\subsection{Case of independent increments multistable L\'evy motion}

Let us first recall the following result (Theorem 4.14 in chapter II of \cite{JS}):

\begin{theo}\label{jacshy}
Let X be a process with independent increments. 
Then X is also a semimartingale if and only if, for each real $u$, 
the function $t \mapsto \E\left(\exp(i u X_t) \right)$ has finite variation over finite 
intervals.
\end{theo}

\begin{prop}
$L_{I}$ is a semi-martingale. 
\end{prop}

\begin{proof}
Let us check that the conditions of Theorem \ref{jacshy} are fulfilled. 
Since $L_{I}$ is an independent increments process, it suffices to verify 
that the function
$t \mapsto \mathbb{E}\left(\exp(i\theta Y(t))\right)$ has finite variations on finite intervals for each $\theta$. This is obvious in view of \eqref{mcfii}.
\end{proof}

It is easy to check that the system of generating triplets (in the notation of \cite{Bk_Sato}) of $L_{I}$ is
$(0,\nu,0)$, with $\nu(dx,dz)= |z|^{-\alpha(x)-1}dzdx$. If $\alpha$ ranges in $[c,d] \subset(0,1)$, 
then $L_{I}$ is a finite variation process, while if $\alpha$ ranges in  
$[c,d] \subset (1,2)$, then $\int_0^1\int_1^{\infty}|z|\nu(dx,dz)=\int_0^1\int_1^{\infty}|z|^{-\alpha(x)}dxdz <\infty$, which implies that $L_{I}$ is a martingale. The proof is a simple adaptation of the one for stable processes and is left to the
reader. In general, the following decomposition holds:
\begin{prop}\label{semimartingale}
\begin{equation*}
L_{I}(t) = A(t) + M(t),
\end{equation*}
where 
$$A(t) = \sum_{(X,Y)\in \Pi, |Y| < 1}C_{\alpha(X)}^{1/ \alpha(X)}\mathbf{1}_{[0,t]}(X)Y^{<-1/\alpha(X)>}$$
is a finite variation process and
$$M(t) = \sum_{(X,Y)\in \Pi, |Y| \geq 1}C_{\alpha(X)}^{1/ \alpha(X)}\mathbf{1}_{[0,t]}(X)Y^{<-1/\alpha(X)>}$$
is a martingale.
\end{prop}

\begin{proof}

That $A$ has finite variation is a direct consequence of the fact that it is almost surely composed of a finite number of (jump) terms.

\noindent To prove that $M$ is a martingale, it is sufficient to show that it is in $L_1(\Omega)$ for all $t$. Note that the jumps of $M$ are bounded by $K=\sup_{b\in[c,d]} C_{b}^{1/ b}$. As a consequence, 
with obvious notation, 
$$\int_0^t\int_1^{\infty}|z|\nu_M(dx,dz) \leq \int_0^t\int_1^{K}|z|\nu(dx,dz) = 
2\int_0^t\int_1^{K}|z|^{-\alpha(x)}dx \ dz < \infty.$$
This implies that, for all $t$, $M(t)$ has finite mean (see, {\it e.g.} \cite[Theorem 25.3]{Bk_Sato}).

\end{proof}
{\bf Remarks}
\begin{itemize}
\item A corresponding decomposition holds of course for the Fergusson-Klass-LePage representation, {\it i.e.}:
$$L_{I}(t)=A'(t)+M'(t),$$
where 
$$A'(t)=\sum_{i=1, \Gamma_i < 1 }^{\infty} C_{\alpha(U_i)}^{1/ \alpha(U_i)} \gamma_i \Gamma_i^{-1/\alpha(U_i)}\mathbf{1}_{(U_i \leq t)}$$ 
has finite variation and 
$$M'(t)=\sum_{i=1, \Gamma_i \geq 1 }^{\infty} C_{\alpha(U_i)}^{1/ \alpha(U_i)} \gamma_i \Gamma_i^{-1/\alpha(U_i)}\mathbf{1}_{(U_i \leq t)} $$ 
is a martingale. The notation $\sum\limits_{i=1, \Gamma_i < 1 }^{\infty}$
means that we sum over the set of indices $\{ i \in \mathbb{N}: \Gamma_i < 1 \}$.
\item As is well-known, the decomposition above is not unique. Another decomposition of interest is the following:

$$L_{I}(t)=M_1(t)+A_1(t)$$
with
$$M_1(t)=\sum_{i=1, \alpha(U_i)>\frac{1}{i} }^{\infty} C_{\alpha(U_i)}^{1/ \alpha(U_i)} \gamma_i \Gamma_i^{-1/\alpha(U_i)}\mathbf{1}_{(U_i \leq t)}$$
is a martingale and 
$$A_1(t)=\sum_{i=1, \alpha(U_i)<\frac{1}{i}}^{\infty} C_{\alpha(U_i)}^{1/ \alpha(U_i)} \gamma_i \Gamma_i^{-1/\alpha(U_i)}\mathbf{1}_{(U_i \leq t)}$$
is an adapted process with finite variations.
\end{itemize}

\subsection{Case of field-based multistable L\'evy motion}
$L_{F}$ is not an independent increments process, and, contrarily to what its definition might suggest 
at first sight, it is not a pure jump process (see below for a more precise statement). 
Thus it is not immediately obvious that it is indeed a semi-martingale. We shall use the characterization 
of semi-martingales as  ``good integrators'' (see, {\it e.g.}
\cite{Bichteler,Protter}) to prove this fact. More precisely, fix $t > 0$ 
and consider simple predictable processes $\xi$ of the form
$$\xi(u) = \xi_0 \mathbf{1}_{\{0\} } (u) + \sum_{k=1}^{n}\limits \xi_k \mathbf{1}_{(s_k, t_k ]} (u)$$
where $0 \leq s_1 < t_1 \leq s_2 < t_2 \leq ... \leq s_n < t_n=t$, $\xi_k \in \mathcal{F}_{s_k}$ and $|\xi_k| \leq 1$ a.s. for all $0 \leq k \leq n$. The integral of $\xi$ with respect to a process $Y$ is:
$$I_Y(\xi)  = \sum_{k=1}^n\limits \xi_k (Y_{t_k} - Y_{s_k} ).$$
It is well known that $Y$ is a semi-martingale if and only if the family $\{ I_Y(\xi), |\xi| \leq 1, \xi \textrm{ is a simple predictable process} \}$ is bounded in probability.

In fact, we shall establish the semi-martingale property for a more general class of processes 
defined through Ferguson-Klass-LePage representations as follows:
\begin{equation}\label{FKL}
X(t)=C^{1/\alpha(t)}_{\alpha(t)} \sum_{i=1}^{\infty} \gamma_i \Gamma_i^{-1/\alpha(t)} f(t,V_i),
\end{equation}
where the function $f$ is such that, for all $t$, $\int_0^T |f(t,x)|^{\alpha(t)} \ dx < \infty$, and give conditions on the kernel $f$ ensuring that $X$ is a semi-martingale. 

\begin{theo}\label{semimgal}
Let $X$ be defined by \eqref{FKL}, with $\alpha$ a $C^1$ function. Assume that $X$ is a c\`adl\`ag adapted process, and that there exists a constant $L_{\infty}$ such that for all $(t,x)  \in \mathbb{R} \times [0,T]$, $|f(t,x)| \leq L_{\infty}$. 
Assume in addition that, for all $x\in E$, the function $u \mapsto f(u,x)$ has finite variation 
on finite intervals, with total variation $Vf(.,x)$ verifying
\begin{equation}\label{intvar}
\int_0^T \left| Vf(.,x) \right|^p dx <+\infty
\end{equation}
for some $p \in (d,2)$.
Then $X$ is a semi-martingale.

\end{theo}

\begin{proof}
We shall show that $\frac{X(t)}{C^{1/\alpha(t)}_{\alpha(t)}}$ is a semi-martingale. Both processes
$t \mapsto \gamma_1 \Gamma_1^{-1/ \alpha(t)}$ and $t \mapsto f(t,V_1)$ have finite variation, thus the
same holds for $t \mapsto \gamma_1 \Gamma_1^{-1/ \alpha(t)} f(t,V_1)$. 
We shall prove that $Y(t)=\sum_{i=2}^{\infty}\limits \gamma_i \Gamma_i^{-1/\alpha(t)} f(t,V_i)$ is a semi-martingale.

One computes 
\begin{eqnarray*}
 I_Y(\xi) & = & \sum_{k=1}^n\limits \xi_k (Y_{t_k} - Y_{s_k} )\\
  & = &  \sum_{k=1}^n\limits \xi_k \sum_{i=2}^{+\infty}\limits \gamma_i \left( \Gamma_i^{-1/ \alpha(t_k)} f(t_k,V_i) - \Gamma_i^{-1/ \alpha(s_k)} f(s_k,V_i)\right) \\
  & = & \sum_{k=1}^n\limits \xi_k \sum_{i=2}^{+\infty}\limits \gamma_i \left( \Gamma_i^{-1/ \alpha(t_k)}  - \Gamma_i^{-1/ \alpha(s_k)} \right) f(t_k,V_i)\\
  & & + \sum_{k=1}^n\limits \xi_k \sum_{i=2}^{+\infty}\limits \gamma_i \Gamma_i^{-1/ \alpha(s_k)} \left( f(t_k,V_i) - f(s_k,V_i)\right)\\
  & =: & A+B.\\
\end{eqnarray*}
Let $K > 0$. We need to control $P(|I_Y(\xi)| >K)$. 

\noindent For the term $A$, we shall use the mean value theorem: there exists a sequence of random variables $w_i^k \in (s_k,t_k)$ such that

$$\Gamma_i^{-1/ \alpha(t_k)}  - \Gamma_i^{-1/ \alpha(s_k)} = (t_k - s_k) \frac{\alpha'(w_i^k)}{\alpha^2 (w_i^k)} (\log \Gamma_i) \Gamma_i^{-1 / \alpha(w_i^k)}.$$

\noindent One has 

\begin{eqnarray*}
 & & P\left( \left|\sum_{k=1}^n\limits \xi_k \sum_{i=2}^{+\infty}\limits \gamma_i \left( \Gamma_i^{-1/ \alpha(t_k)}  - \Gamma_i^{-1/ \alpha(s_k)} \right) f(t_k,V_i) \right| >K \right) \\
 &= & P\left( \left|\sum_{i=2}^{+\infty}\limits \gamma_i \sum_{k=1}^n\limits \xi_k \left( (t_k - s_k) \frac{\alpha'(w_i^k)}{\alpha^2 (w_i^k)} (\log \Gamma_i) \Gamma_i^{-1 / \alpha(w_i^k)} \right) f(t_k,V_i) \right|^p >K^p \right) \\
&  \leq & \frac{1}{K^p} E \left[ \left|\sum_{i=2}^{+\infty}\limits \gamma_i \sum_{k=1}^n\limits \xi_k \left( (t_k - s_k) \frac{\alpha'(w_i^k)}{\alpha^2 (w_i^k)} (\log \Gamma_i) \Gamma_i^{-1 / \alpha(w_i^k)} \right) f(t_k,V_i) \right|^p \right].\\
 \end{eqnarray*}
As in the proof of Proposition 4.9 of \cite{LL2}, we use Theorem 2 of \cite{BE} (note that random variables $w_i^k$ are independent of the 
$(\gamma_j)_j$). Since $p$<2,

\begin{eqnarray*}
 & & P\left( \left|\sum_{k=1}^n\limits \xi_k \sum_{i=2}^{+\infty}\limits \gamma_i \left( \Gamma_i^{-1/ \alpha(t_k)}  - \Gamma_i^{-1/ \alpha(s_k)} \right) f(t_k,V_i) \right| >K \right) \\
 & \leq & \frac{2}{K^p} \sum_{i=2}^{+\infty}\limits E \left[ \left|\gamma_i \sum_{k=1}^n\limits \xi_k \left( (t_k - s_k) \frac{\alpha'(w_i^k)}{\alpha^2 (w_i^k)} (\log \Gamma_i) \Gamma_i^{-1 / \alpha(w_i^k)} \right) f(t_k,V_i) \right|^p \right] \\
  & \leq & \frac{2}{K^p} \sum_{i=2}^{+\infty}\limits E \left[\left( \sum_{k=1}^n\limits (t_k - s_k) \frac{\sup_{s \in [ 0,t ]} |\alpha'(s)|}{c^2} \left|\log \Gamma_i\right| ( \Gamma_i^{-1 / c} +\Gamma_i^{-1 / d} ) L_{\infty} \right)^p \right] \\
  & \leq & 2 \left(\frac{L_{\infty} t \sup_{s \in [ 0,t ]} |\alpha'(s)|}{K c^2}\right)^p \ \sum_{i=2}^{+\infty}\limits E \left[ \left|\log \Gamma_i ( \Gamma_i^{-1 / c} +\Gamma_i^{-1 / d} )\right|^p \right],\\
 \end{eqnarray*}
where the infinite sum in the last line above converges since $p>d$.
Let us now consider the second term $B$ of $I_Y(\xi)$ :
\begin{eqnarray*}
 & & P\left( \left|\sum_{k=1}^n\limits \xi_k \sum_{i=2}^{+\infty}\limits \gamma_i \Gamma_i^{-1/ \alpha(s_k)} \left( f(t_k,V_i) - f(s_k,V_i)\right) \right| >K \right) \\
& \leq & \frac{1}{K^p} E \left[ \left|\sum_{i=2}^{+\infty}\limits \gamma_i  \sum_{k=1}^n\limits \xi_k \Gamma_i^{-1/ \alpha(s_k)} \left( f(t_k,V_i) - f(s_k,V_i) \right) \right|^p \right]. \\
\end{eqnarray*}
We use again Theorem 2 of \cite{BE} :
\begin{eqnarray*}
 & & P\left( \left|\sum_{k=1}^n\limits \xi_k \sum_{i=2}^{+\infty}\limits \gamma_i \Gamma_i^{-1/ \alpha(s_k)} \left( f(t_k,V_i) - f(s_k,V_i)\right) \right| >K \right) \\
& \leq & \frac{2}{K^p} \sum_{i=2}^{+\infty}\limits E \left[ \left|\sum_{k=1}^n\limits \xi_k \Gamma_i^{-1/ \alpha(s_k)} \left( f(t_k,V_i) - f(s_k,V_i) \right) \right|^p \right] \\
& \leq & \frac{2}{K^p} \sum_{i=2}^{+\infty}\limits E \left[ \left( \sum_{k=1}^n\limits \left( \Gamma_i^{-1/ c} + \Gamma_i^{-1/d}\right) \left| f(t_k,V_i) - f(s_k,V_i) \right| \right)^p \right] \\
& \leq & \frac{2}{K^p} \sum_{i=2}^{+\infty}\limits E \left[ \left( \Gamma_i^{-1/ c} + \Gamma_i^{-1/d}\right)^p \left( \sum_{k=1}^n\limits \left| f(t_k,V_i) - f(s_k,V_i) \right| \right)^p \right] \\
& \leq & \frac{2}{K^p} \sum_{i=2}^{+\infty}\limits E \left[ \left( \Gamma_i^{-1/ c} + \Gamma_i^{-1/d}\right)^p |Vf(.,V_i)|^p \right]\\
& = & \frac{2}{K^p} \sum_{i=2}^{+\infty}\limits E \left[ \left( \Gamma_i^{-1/ c} + \Gamma_i^{-1/d}\right)^p \right] E \left[|Vf(.,V_i)|^p \right]\\
& = & \frac{2}{K^p} E \left[|Vf(.,V_1)|^p \right] \sum_{i=2}^{+\infty}\limits E \left[ \left( \Gamma_i^{-1/ c} + \Gamma_i^{-1/d} \right)^p \right],\\
\end{eqnarray*}
where, in the last line above, the first expectation is finite by assumption \eqref{intvar} and the infinite sum converges converges because $p>d$. 

We have thus shown that $P(|I_Y(\xi)| >K) \leq \frac{C}{K^p}$ for some constant $C$, as required.
\end{proof}

\begin{cor}
Assume $\alpha$ is a $C^1$ function. Then the field-based multistable L\'evy motion is a semi-martingale.
\end{cor}

The result above does not give a semi-martingale decomposition of $L_{F}$. 
The following theorem does so. In addition, it elucidates the links between the two multistable L\'evy motions:

\begin{theo}
Assume $\alpha$ is a $C^1$ function. Almost surely, for all $t$,
\begin{equation}\label{decomp}
L_{F}(t) = A(t) + L_{I}(t),
\end{equation}
where $A$ is the finite variation process defined by:
\begin{equation}
A(t)=\int_0^t \sum_{i=1}^{+\infty} \gamma_i \frac{d\left(C_{\alpha(u)}^{1/\alpha(u)} \Gamma_i^{-1/\alpha(u)}\right)}{du}(s) \mathbf{1}_{[0,s[}(V_i) ds.
\end{equation}
\end{theo}
From a heuristic point of view, this result states that, at all jumps points, both L\'evy multistable motions vary by the same amount. ``In-between''
jumps, however, $L_{I}$ does not change (it is a pure jump process), while $L_{F}$ moves in a continuous fashion (and thus is not a pure jump process).

 \begin{proof}
%
%
 We wish to prove that:
\begin{displaymath}
C_{\alpha(t)}^{1/\alpha(t)}\sum_{i=1}^{+\infty} \gamma_i \Gamma_i^{-1/\alpha(t)}\mathbf{1}_{[0,t]}(V_i) = \int_0^t \sum_{i=1}^{+\infty} \gamma_i g'_i(s) \mathbf{1}_{[0,s[}(V_i) ds + \sum_{i=1}^{+\infty}\gamma_i C_{\alpha(V_i)}^{1/\alpha(V_i)} \Gamma_i^{-1/\alpha(V_i)}\mathbf{1}_{[0,t]}(V_i),
\end{displaymath}
where $g_i(t) = C_{\alpha(t)}^{1/\alpha(t)}\Gamma_i^{-1/\alpha(t)}.$

We first prove a lemma about the uniform convergence of the series under consideration:

\begin{lem}\label{LemUnifConv}
 Let $(U_i^{(k)})_{i \in \bbbn}$, $k \in \{1,2\}$, two sequences satisfying for all $i \in \bbbn^*$, $U_i^{(1)} = 1$ and $U_i^{(2)} = \ln \Gamma_i$.
 
 Let $$D_N^{(k)}(s) = \sum_{i=1}^{N} \gamma_i U_i^{(k)} \Gamma_i^{-1/\alpha(s)} \mathbf{1}_{[0,s[}(V_i),$$
 and
 $$D^{(k)}(s) = \sum_{i=1}^{+\infty} \gamma_i U_i^{(k)} \Gamma_i^{-1/\alpha(s)} \mathbf{1}_{[0,s[}(V_i). $$
 
 Then, for all $k \in \{1,2\}$, $D_N^{(k)}(s)$ converges uniformly on $[0,1]$ to $D^{(k)}(s)$.
\end{lem}

{\it Proof of Lemma \ref{LemUnifConv}:}

We introduce the series $E_N^{(k)}(s) = \sum_{i=1}^{N}\limits \gamma_i W_i^{(k)} i^{-1/\alpha(s)} \mathbf{1}_{[0,s[}(V_i)$ and 
 $E^{(k)}(s) = \sum_{i=1}^{+\infty}\limits \gamma_i W_i^{(k)} i^{-1/\alpha(s)} \mathbf{1}_{[0,s[}(V_i),$
 where $W_i^{(1)} = U_i^{(1)}$ and $W_i^{(2)} = \ln (i)$.
 
 We first prove that $E_N^{(k)}(s)$ converges uniformly on $[0,1]$ to $E^{(k)}(s)$.
 
  \noindent {\bf Case $d<1$:}
  
  \noindent Fix $1> \hat{d} > d$ and $i_0 \in \mathbb{N}$ such that, for $i \geq i_0$, $|W_i^{(k)}| \leq \ln(i)$ and $\frac{\ln(i)}{i^{1/ d}} \leq \frac{1}{i^{1/ \hat{d}}}$.
 Then, for all $s \in [0,1]$ and $q \geq p$,
 \begin{eqnarray*}
\left| \sum_{i=p}^{q} \gamma_i W_i^{(k)} i^{-1/ \alpha(s)} \mathbf{1}_{[0,s[}(V_i)\right| & \leq & \sum_{i=p}^{q} \frac{|W_i^{(k)}|}{i^{1/ d}} \\
 & \leq & \sum_{i=p}^{q} \frac{1}{i^{1/ \hat{d}}}\\
 & \leq & \sum_{i=p}^{+\infty}  \frac{1}{i^{1/ \hat{d}}}\\
 & \leq & p^{1-1/ \hat{d}}.
 \end{eqnarray*}
 The uniform Cauchy criterion implies that $E_N^{(k)}(s)$ converges uniformly on $[0,1]$ to $E^{(k)}(s)$.
 
 \noindent {\bf Case $d\geq 1$:}
 
 \noindent Set $$E_{\alpha}^j(s) = \sum_{i=2^{j}}^{2^{j+1}-1} \gamma_i W_i^{(k)} i^{-1/ \alpha(s)} \mathbf{1}_{[0,s[}(V_i),$$
 and $b(j) = \frac{j(j+1)\ln 2}{2^{j/d}} \sqrt{2^j}.$
Consider 
 
\begin{eqnarray*}
E(j) &=& \left\{ \sup_{s \in [0,1]} |E_{\alpha}^j(s)| \leq \frac{j(j+1)\ln 2}{2^{j/d}} \sqrt{2^j}\right\}\\
& = &\left\{ \sup_{s \in [0,1]} |E_{\alpha}^j(s)| \leq b(j)\right\}.\\
\end{eqnarray*}
Fix $\hat{d} > d \geq 1$ and $i_0 \in \mathbb{N}$ such that, for $i \geq i_0$, $\frac{(\ln i)^2}{i^{1/ d}} \leq \frac{1}{i^{1/ \hat{d}}}.$

 Let $\delta > \frac{1}{2} + \frac{1}{d} - \frac{1}{\hat{d}}.$
 Define the step function $\alpha_j(s) = \sum_{k=0}^{[2^{j \delta}]-1}\limits \alpha(\frac{k}{[2^{j \delta}]}) \mathbf{1}_{[ \frac{k}{[2^{j \delta}]}, \frac{k+1}{[2^{j \delta}]}[}(s).$ 
For all $s \in [0,1]$,
\begin{eqnarray*}
|\alpha_j(s) - \alpha(s)| & = & \sum_{k=0}^{[2^{j \delta}]-1} | \alpha(\frac{k}{[2^{j \delta}]}) - \alpha(s) | \mathbf{1}_{[ \frac{k}{[2^{j \delta}]}, \frac{k+1}{[2^{j \delta}]}[}(s) \\
& \leq & \frac{\sup_{t \in [0,1]}\limits |\alpha'(t)|}{[2^{j \delta}]},\\
\end{eqnarray*}
and thus $$\sup_{s \in [0,1]} |\alpha_j(s) - \alpha(s)| \leq \frac{A}{[2^{j \delta}]},$$
with $A=\sup_{t \in [0,1]}\limits |\alpha'(t)|.$

Set 
 $$E_{\alpha_j}^j(s) = \sum_{i=2^{j}}^{2^{j+1}-1}\limits \gamma_i W_i^{(k)}i^{-1/\alpha_j(s)} \mathbf{1}_{[0,s[}(V_i).$$
 Then,
 \begin{eqnarray*}
 E_{\alpha}^j(s) - E_{\alpha_j}^j(s) & = &  \sum_{i=2^{j}}^{2^{j+1}-1} \gamma_i W_i^{(k)} \left(i^{-1/\alpha(s)} - i^{-1/\alpha_j(s)} \right) \mathbf{1}_{[0,s[}(V_i). \\
 \end{eqnarray*}
 The finite increments theorem implies that there exists $\alpha^i(s) \in [\alpha_j(s), \alpha(s)]$ (or in $[\alpha(s), \alpha_j(s)]$) such that
 \begin{eqnarray*}
E_{\alpha}^j(s) - E_{\alpha_j}^j(s) & = & (\alpha(s) - \alpha_j(s) ) \sum_{i=2^{j}}^{2^{j+1}-1} \gamma_i \frac{\ln i}{(\alpha^i(s))^2} W_i^{(k)} i^{-1/\alpha^i(s)} \mathbf{1}_{[0,s[}(V_i), \\
\end{eqnarray*}
and thus
 \begin{eqnarray*}
 |E_{\alpha}^j(s) - E_{\alpha_j}^j(s) | & \leq & |\alpha(s) - \alpha_j(s)|  \sum_{i=2^{j}}^{2^{j+1}-1} \frac{(\ln i)^2 }{d^2} i^{-1/d}\\
  & \leq & \frac{A}{[2^{j \delta}] d^2} \sum_{i=2^{j}}^{2^{j+1}-1}\frac{(\ln i)^2}{i^{1/d}}.\\
  \end{eqnarray*}

 There exists $J_0 \in \mathbb{N}$ such that, for all $j \geq J_0$,
 \begin{eqnarray*}
 \sup_{s \in [0,1]} |E_{\alpha}^j(s) - E_{\alpha_j}^j(s) | & \leq & \frac{b(j)}{2}.\\
 \end{eqnarray*} 
 Then, for $j \geq J_0$, 
   \begin{eqnarray*}
   \P \left( \overline{E(j)} \right) & = & \P \left( \sup_{s \in [0,1]} |E_{\alpha}^j(s)| > b(j) \right) \\
   & \leq & \P \left( \frac{b(j)}{2} + \sup_{s \in [0,1]} |E_{\alpha_j}^j(s)| > b(j) \right)\\
   & =  &\P \left(\sup_{s \in [0,1]} |E_{\alpha_j}^j(s)| > \frac{b(j)}{2} \right) .\\
   \end{eqnarray*}

 Denote by $\mathcal{A}^j$ the set of possible values of $\alpha_j$. Then $Card(\mathcal{A}^j)=[2^{j \delta}].$  
For each random drawing of the $(V_i)_i$, $E_{\alpha_j}^j(s)$ is composed of a sum of $n$ terms of the form $\gamma_{l_i} W_{l_i}^{(k)} l_i^{-1/ \alpha_0}$ where $\alpha_0 \in \mathcal{A}^j$. There are $2^j$ possible values for $n$, and $[2^{j \delta}]$ possible values for $\alpha_0$. One thus has
the following estimates:
 \begin{eqnarray*}
\hspace{-1cm} \P \left( \overline{E(j)} \right)  & \leq & \P \left(\sup_{s \in [0,1]} |E_{\alpha_j}^j(s)| > \frac{b(j)}{2} \right) \\
 & \leq & \P \left( \bigcup_{l_1,...,l_{2^j} \in \llbracket 2^j,2^{j+1}-1 \rrbracket} \{ V_{l_1} < V_{l_2} < ... < V_{l_{2^j}} \} \cap \{ \sup_{ \alpha_0 \in \mathcal{A}^j} \sup_{1\leq m \leq 2^j} |\sum_{i=1}^{m} \gamma_{l_i} W_{l_i}^{(k)} l_i^{-1/ \alpha_0} | > \frac{b(j)}{2}\} \right)\\
 & \leq & \sum_{l_1,...,l_{2^j} \in \llbracket 2^j,2^{j+1}-1 \rrbracket}\P \left(  \{ V_{l_1} < V_{l_2} < ... < V_{l_{2^j}} \} \cap \{ \sup_{ \alpha_0 \in \mathcal{A}^j} \sup_{1\leq m \leq 2^j} |\sum_{i=1}^{m} \gamma_{l_i} W_{l_i}^{(k)} l_i^{-1/ \alpha_0} | > \frac{b(j)}{2}\} \right) \\
 & \leq & \sum_{l_1,...,l_{2^j} \in \llbracket 2^j,2^{j+1}-1 \rrbracket} \P \left(  V_{l_1} < V_{l_2} < ... < V_{l_{2^j}} \right) \P \left( \sup_{ \alpha_0 \in \mathcal{A}^j} \sup_{1\leq m \leq 2^j} |\sum_{i=1}^{m} \gamma_{l_i} W_{l_i}^{(k)} l_i^{-1/ \alpha_0} | > \frac{b(j)}{2} \right)\\
 & \leq & \sum_{l_1,...,l_{2^j} \in \llbracket 2^j,2^{j+1}-1 \rrbracket} \frac{1}{(2^j)!} \P \left( \cup_{ \alpha_0 \in \mathcal{A}^j} \cup_{m =1}^{2^j} |\sum_{i=1}^{m} \gamma_{l_i} W_{l_i}^{(k)} l_i^{-1/ \alpha_0} | > \frac{b(j)}{2} \right)\\
 & \leq &  \frac{1}{(2^j)!}\sum_{l_1,...,l_{2^j} \in \llbracket 2^j,2^{j+1}-1 \rrbracket}  \sum_{ \alpha_0 \in \mathcal{A}^j} \sum_{m =1}^{2^j} \P \left(|\sum_{i=1}^{m} \gamma_{l_i} W_{l_i}^{(k)} l_i^{-1/ \alpha_0} | > \frac{b(j)}{2} \right) \\
 & \leq &  \frac{1}{(2^j)!}\sum_{l_1,...,l_{2^j} \in \llbracket 2^j,2^{j+1}-1 \rrbracket}  \sum_{ \alpha_0 \in \mathcal{A}^j} \sum_{m =1}^{2^j} \P \left(|\sum_{i=1}^{m} \gamma_{l_i} \frac{W_{l_i}^{(k)}}{(j+1)\ln 2} \frac{2^{j/d}}{l_i^{1/ \alpha_0}} | > \frac{j}{2} \sqrt{2^j}\right) \\
 & \leq &  \frac{1}{(2^j)!}\sum_{l_1,...,l_{2^j} \in \llbracket 2^j,2^{j+1}-1 \rrbracket}  \sum_{ \alpha_0 \in \mathcal{A}^j} \sum_{m =1}^{2^j} 2e^{-\frac{j^2}{8}}\\
 & \leq & 2.2^j[2^{j \delta}]e^{-\frac{j^2}{8}},\\
 \end{eqnarray*} 
 where Lemma 1.5 of \cite{LTal} was used in the end.
 As a consequence, $$\P \left( \liminf_j E(j) \right) =1.$$
  
Now let $p \in \mathbb{N}$, and set $j_p = [\frac{\ln p}{\ln 2} ]$. Define  for $m \in \llbracket 2^{j_p}, 2^{j_p+1}-1 \rrbracket$
 $$E_{\alpha}^{m,p}(s) = \sum_{i=m}^{2^{j_p+1}-1} \gamma_i W_i^{(k)} i^{-1/\alpha(s)} \mathbf{1}_{[0,s[}(V_i),$$ 
 and 
 $$E_{\alpha_{j_p}}^{m,p}(s) =  \sum_{i=m}^{2^{j_p+1}-1} \gamma_i W_i^{(k)} i^{-1/\alpha_{j_p}(s)} \mathbf{1}_{[0,s[}(V_i).$$
 As above,
 $$\sup_{s \in [0,1]} |\alpha_{j_p}(s) - \alpha(s)| \leq \frac{A}{[2^{j_p \delta}]} $$ 
 and
 \begin{eqnarray*}
 |E_{\alpha}^{m,p}(s) - E_{\alpha_{j_p}}^{m,p}(s) | & \leq & \frac{A}{[2^{j_p \delta}] d^2} \sum_{i=m}^{2^{j_p+1}-1}\frac{1}{i^{1/ \hat{d}}}\\
 & \leq & \frac{A}{[2^{j_p \delta}] d^2} \sum_{i=2^{j_p}}^{2^{j_p+1}-1}\frac{1}{i^{1/ \hat{d}}}\\
 & \leq &  \frac{A}{[2^{j_p \delta}] d^2} \frac{2^{j_p}}{2^{j_p/ \hat{d}}}.\\
 \end{eqnarray*}
 Fix $p_0 \in \mathbb{N}$ such that for all $p \geq p_0$,
 \begin{eqnarray*}
 \sup_{m \in \llbracket 2^{j_p}, 2^{j_p+1}-1 \rrbracket } \sup_{s \in [0,1]} |E_{\alpha}^{m,p}(s) - E_{\alpha_{j_p}}^{m,p}(s) | & \leq & \frac{b(j_p)}{2},\\
 \end{eqnarray*} 
 and consider 
 \begin{eqnarray*}
 E(p) &=& \left\{ \sup_{m \in \llbracket 2^{j_p}, 2^{j_p+1}-1 \rrbracket } \sup_{s \in [0,1]} |E_{\alpha}^{m,p}(s)| \leq b(j_p)\right\}.\\
 \end{eqnarray*}
 For $p \geq p_0$,
 \begin{eqnarray*}
 \P \left( \overline{E(p)} \right) & \leq & \P \left( \sup_{m \in \llbracket 2^{j_p}, 2^{j_p+1}-1 \rrbracket } \sup_{s \in [0,1]} |E_{\alpha_{j_p}}^{m,p}(s)| > \frac{b(j_p)}{2}\right)\\
 & \leq & \sum_{m \in \llbracket 2^{j_p}, 2^{j_p+1}-1 \rrbracket } \P \left( \sup_{s \in [0,1]} |E_{\alpha_{j_p}}^{m,p}(s)| > \frac{b(j_p)}{2}\right)\\
 & \leq & \sum_{m \in \llbracket 2^{j_p}, 2^{j_p+1}-1 \rrbracket } 2.2^{j_p}[2^{j_p \delta}]e^{-\frac{j_p^2}{8}}\\
 & \leq & 2.2^{2j_p}[2^{j_p \delta}]e^{-\frac{j_p^2}{8}}.\\
 \end{eqnarray*}
 As a consequence,
 $$\P \left( \liminf_p E(p) \right) =1.$$
 Reasoning along the same lines and setting 
 $$E_{\alpha}^{m,q}(s) = \sum_{i=2^{j_q}}^{m} \gamma_i W_i^{(k)} i^{-1/\alpha(s)} \mathbf{1}_{[0,s[}(V_i),$$ 
 and 
 $$E(q) = \left\{ \sup_{m \in \llbracket 2^{j_q}, 2^{j_q+1}-1 \rrbracket } \sup_{s \in [0,1]} |E_{\alpha}^{m,q}(s)| \leq b(j_q)\right\},$$
 one gets $$\P \left( \liminf_q E(q) \right) =1.$$
 
 Finally, let $p,q \in \mathbb{N}$, with $q \geq p$, and $p$ large enough. For all $s \in [0,1]$,
 \begin{eqnarray*}
 \left| E_q(s) - E_p(s) \right| & = & \left| \sum_{i=p}^{2^{j_p+1}-1} \gamma_i W_i^{(k)} i^{-1/\alpha(s)} \mathbf{1}_{[0,s[}(V_i) + \sum_{j=j_p +1}^{j_q -1} \sum_{i=2^j}^{2^{j+1}-1} \gamma_i W_i^{(k)} i^{-1/\alpha(s)} \mathbf{1}_{[0,s[}(V_i) \right. \\
& & \left. +\sum_{i=2^{j_q}}^{q} \gamma_i W_i^{(k)} i^{-1/\alpha(s)} \mathbf{1}_{[0,s[}(V_i) \right|\\
  & \leq & b(j_p) + \sum_{j=j_p +1}^{j_q -1} b(j) + b(j_q)\\
 & \leq & 2 b(j_p) + \sum_{j=j_p}^{+ \infty} b(j) .\\
 \end{eqnarray*}
 Again, the uniform Cauchy criterion implies that $E_N(s)$ converges uniformly on $[0,1]$ to $E(s)$.
 
\noindent {\bf Convergence of the difference term:}

\noindent Let $q \geq p$ and denote $R_{p,q}(s) = \sum_{i=p}^{q}\limits \gamma_i (U_i^{(k)} \Gamma_i^{-1/\alpha(s)} - W_i^{(k)} i^{-1/ \alpha(s)} )\mathbf{1}_{[0,s[}(V_i).$

\noindent Fix $p_0$ such that, for all $p \geq p_0$, 
$$R_{p,q}(s) =  \sum_{i=p}^{q}\limits \gamma_i (U_i^{(k)} \Gamma_i^{-1/\alpha(s)} - W_i^{(k)} i^{-1/ \alpha(s)} )\mathbf{1}_{[0,s[}(V_i)\mathbf{1}_{\frac{1}{2} \leq \frac{\Gamma_i}{i} \leq 2}.$$
One computes: 
$$ | R_{p,q}(s) | \leq  \sum_{i=p}^{q} |U_i^{(k)} \Gamma_i^{-1/\alpha(s)} - W_i^{(k)} i^{-1/ \alpha(s)}|\mathbf{1}_{\frac{1}{2} \leq \frac{\Gamma_i}{i} \leq 2}.$$

\noindent \underline{Case $k=1$:}

For $i \geq 1$, $U_i^{(1)}=W_i^{(1)}$. We then obtain
\begin{eqnarray*}
| R_{p,q}(s) | &\leq & \sum_{i=p}^{q} |i^{-1/ \alpha(s)}-\Gamma_i^{-1/ \alpha(s)} |\mathbf{1}_{\frac{1}{2} \leq \frac{\Gamma_i}{i} \leq 2}\\
& \leq & \sum_{i=p}^{q} \frac{1}{i^{1/d}} |1-(\frac{\Gamma_i}{i})^{-1/ \alpha(s)} |\mathbf{1}_{\frac{1}{2} \leq \frac{\Gamma_i}{i} \leq 2}.
\end{eqnarray*}
There exists a positive constant $K_{c,d}$ which depends only on $c$ and $d$ such that
$$| R_{p,q}(s) | \leq  K_{c,d} \sum_{i=p}^{q} \frac{1}{i^{1/d}} \left| \frac{\Gamma_i}{i} -1 \right |.$$

\noindent \underline{Case $k=2$:}

One computes:

\begin{eqnarray*}
| R_{p,q}(s) | &\leq & \sum_{i=p}^{q} | \ln(i) i^{-1/ \alpha(s)}- \ln(\Gamma_i) \Gamma_i^{-1/ \alpha(s)} |\mathbf{1}_{\frac{1}{2} \leq \frac{\Gamma_i}{i} \leq 2}\\
&\leq & \sum_{i=p}^{q} \frac{\ln(i)}{i^{1/d}} \left| 1 - \frac{\ln(\Gamma_i)}{\ln(i)} (\frac{\Gamma_i}{i} )^{-1/ \alpha(s)} \right|\mathbf{1}_{\frac{1}{2} \leq \frac{\Gamma_i}{i} \leq 2}\\
& \leq & K_{c,d} \sum_{i=p}^{q} \frac{\ln(i)}{i^{1/d}} \left| \frac{\Gamma_i}{i} -1 \right|,
\end{eqnarray*}
where $K_{c,d}$ is again a positive constant.

Finally, in the two cases, $$| R_{p,q}(s) | \leq  K_{c,d} \sum_{i=p}^{q} \frac{\ln(i)}{i^{1/d}} \left| \frac{\Gamma_i}{i} -1 \right|.$$

 The series $\sum_i\limits \frac{\ln i}{i^{1/d}} \left| \frac{\Gamma_i}{i} -1 \right |$ converges almost surely. The uniform Cauchy criterion
 thus applies to the effect that $D_N^{(k)}(s)-E_N^{(k)}(s)$ converges uniformly. As a consequence, $D_N^{(k)}(s)$ converges uniformly to $D^{(k)}(s)$. $\Box$

 We have shown that, almost surely,
 $$\lim_{N \rightarrow +\infty }\int_0^t \sum_{i=1}^{N} \gamma_i g'_i(s) \mathbf{1}_{[0,s[}(V_i) ds = \int_0^t \sum_{i=1}^{+\infty} \gamma_i g'_i(s) \mathbf{1}_{[0,s[}(V_i) ds .$$ 

Besides,

 \begin{eqnarray*}
\int_0^t \sum_{i=1}^{N} \gamma_i g'_i(s) \mathbf{1}_{[0,s[}(V_i) ds & = & \sum_{i=1}^{N} \gamma_i \int_0^t g'_i(s) \mathbf{1}_{[0,s[}(V_i) ds \\
 & = & \sum_{i=1}^{N} \gamma_i \left( \int_{V_i}^t g'_i(s) ds \right) \mathbf{1}_{[0,t]}(V_i)\\
 & = & \sum_{i=1}^{N} \gamma_i \left( g_i(t) - g_i(V_i) \right) \mathbf{1}_{[0,t]}(V_i)\\
 & = & \sum_{i=1}^{N} \gamma_i g_i(t)\mathbf{1}_{[0,t]}(V_i) - \sum_{i=1}^{N} \gamma_i g_i(V_i)\mathbf{1}_{[0,t]}(V_i)
 \end{eqnarray*}
 which implies finally that
$$\lim_{N \rightarrow +\infty }\int_0^t \sum_{i=1}^{N} \gamma_i g'_i(s) \mathbf{1}_{[0,s[}(V_i) ds = \sum_{i=1}^{+\infty} \gamma_i g_i(t)\mathbf{1}_{[0,t]}(V_i) - \sum_{i=1}^{+\infty} \gamma_i g_i(V_i)\mathbf{1}_{[0,t]}(V_i).$$
 This is \eqref{decomp}. 
 
 That $A$ has finite variation follows from the fact that it is an absolutely continuous function.
 \end{proof}

\section*{Acknowledgments}
Support from SMABTP is gratefully acknowledged.

\bibliographystyle{plain}

 \end{document}